\documentclass[11pt]{amsart}
\usepackage{amssymb,latexsym}
\newtheorem{theorem}{Theorem}[section]
\newtheorem{prop}[theorem]{Proposition}
\newtheorem{lemma}[theorem]{Lemma}
\newtheorem{remark}[theorem]{Remark}
\newtheorem{question}[theorem]{Question}
\newtheorem{definition}[theorem]{Definition}

\newtheorem{example}[theorem]{Example}

\begin{document}

\title{Geometric structures, Gromov norm and Kodaira dimensions}
\author{Weiyi Zhang}
\address{Mathematics Institute\\  University of Warwick\\ Coventry, CV4 7AL, England}
\email{weiyi.zhang@warwick.ac.uk}

\begin{abstract}
We define the Kodaira dimension for $3$-dimensional manifolds through Thurston's eight geometries, along with a classification in terms of this Kodaira dimension. We show this is compatible with other existing Kodaira dimensions and the partial order defined by non-zero degree maps. For higher dimensions, we explore the relations of geometric structures and mapping orders with various Kodaira dimensions and other invariants. Especially, we show that a closed geometric $4$-manifold has nonvanishing Gromov norm if  and only if it has geometry $\mathbb H^2\times \mathbb H^2$, $\mathbb H^2(\mathbb C)$ or $\mathbb H^4$. 
\end{abstract}

\maketitle

\tableofcontents

\section{Introduction}
Complex Kodaira dimension $\kappa^h(M,J)$ provides a very successful classification
scheme for complex manifolds. This notion is generalized by several authors ({\it c.f.} \cite{L,McS,McS2,LeBmrl, LeBcag}) to symplectic manifolds, especially of dimension two and four. In these two dimensions, this symplectic Kodaira dimension is independent of the choice of symplectic structures \cite{L}. In other words, it is a smooth invariant of the manifold which is thus denoted by $\kappa^s(M)$. In dimension four, the smaller the symplectic Kodaira dimension, the more we know. Symplectic $4$-manifolds with $\kappa^s=-\infty$ are diffeomorphic to rational or ruled surfaces \cite{Liu}. When $\kappa^s=0$, all known examples are K3 surface, Enrique surface and $T^2$ bundles over $T^2$. Moreover, it is shown in \cite{L} that a symplectic manifold with $\kappa^s=0$ has the same homological invariants as one of the manifolds listed above. When $\kappa^s=1$ or $2$, no classification is possible since symplectic manifolds in both categories could admit arbitrary finitely presented group as their fundamental group \cite{Gom}.

 In \cite{DZ}, the authors prove that
complex and symplectic Kodaira dimensions are compatible with each other. More precisely, when a $4$-manifold $M$ admits at the same time both complex and symplectic structures (but the structures are not necessarily compatible with each other),
then $\kappa^s(M)=\kappa^h(M,J)$. In \cite{LZadd}, a general framework of  ``additivity of Kodaira dimension" is provided to further understand the compatibility of various Kodaira dimensions in possibly different dimensions. In particular, it is shown that the Kodaira dimensions are additive for fiber bundles, Lefschetz fibrations and coverings.

Higher dimensional generalizations of Kodaira dimension,\  {\it e.g.} symplectic Kodaira dimension in dimensions six or higher, are less understood except for a proposed definition in \cite{LR}. Like complex Kodaira dimension, it will no longer be a smooth invariant. Hence, the study of this notion in higher dimensions will be associated to the study of deformation classes of symplectic structures and symplectic birational geometry.

As suggested by the additivity framework, dimension three should also attach certain counterpart of Kodaira dimension. In this paper, we give a definition of Kodaira
dimension $\kappa^t(M)$ in dimension three through Thurston's eight $3$-dimensional geometries and the Geometrization Theorem. Certain classification with respect to $\kappa^t(M)$ is given. This notion is then discussed in the framework of ``additivity of Kodaira dimension".  In this sense, it is compatible with the complex Kodaira dimension and symplectic Kodaira dimension in dimension $4$. Remarkably, we show that the $3$-dimensional Kodaira dimension is compatible with the partial order defined by non-zero degree maps.

\begin{theorem}\label{mapKod3intro}
If $f:M^3\longrightarrow N^3$ is a non-zero degree map, then $\kappa^t(M)\ge \kappa^t(N)$.
\end{theorem}

This result could also be viewed as the first step towards a relative version of $3$-dimensional Kodaira dimension as what we did for $4$-dimensional symplectic manifolds in \cite{LZadd}.

The {\it Gromov norm} (or {\it simplicial volume}) is a homotopy invariant of oriented closed manifolds which is introduced by Gromov in \cite{GroNorm}. It is defined by minimizing the sum of the absolute values of the coefficients over all singular chains representing the fundamental class.
 A remarkable fact in dimension three is that a closed geometric $3$-manifold has nonzero Gromov norm if and only if it is hyperbolic. 

There are nineteen $4$-dimensional geometries, which were classified by Filipkiewicz \cite{Filip}. As in dimension three, we divide the nineteen geometries into $4$ categories: $-\infty, 0, 1$ and $2$, corresponding to the $4$ possible values of Kodaira dimensions of $4$-manifolds. The following fact identifies closed geometric $4$-manifolds with vanishing Gromov norm.

\begin{theorem}\label{0gromovnorm}
A closed geometric $4$-manifold has nonzero Gromov norm if and only if it has geometry $\mathbb H^4$, $\mathbb H^2\times \mathbb H^2$ or $\mathbb H^2(\mathbb C)$.
\end{theorem}

The non-vanishing part of the above theorem is due to \cite{LS, BK, GroNorm}. After posting this paper, the author is kindly informed by Pablo Su\'arez-Serrato that the Theorem \ref{0gromovnorm} has been obtained by him in \cite{S-S}. Actually, \cite{S-S} also establishes other equivalence conditions. Especially, it shows that a closed geometric $4$-manifold not modeled on $\mathbb H^4$, $\mathbb H^2\times \mathbb H^2$ or $\mathbb H^2(\mathbb C)$ if and only if it admits an $\mathcal F$-structure in the sense of Cheeger-Gromov \cite{ChGr}, which implies Gromov norm zero by \cite{ChGr, PPinv}. On the other hand, our argument is shorter and does not involve $\mathcal F$-structures. Instead, we use the $S^1$-action vanishing result of \cite{Yano} and amenable fundamental group vanishing result of \cite{GroNorm}.

For  dimensions greater than or equal to four, we also explore the mapping orders defined by maps regarding to various structures, {\it e.g.}  complex, symplectic or $J$-holomorphic. In the spirit of Theorem \ref{mapKod3intro}, we then discuss the relations of these mapping orders with  Kodaira dimensions and other invariants associated to different structures,\  {\it e.g.} the Gromov norm, topological entropy, $J$-(anti)-invariant cohomology {\it etc.}. Several structural properties of non-zero degree maps and degree one maps are discussed along this line. Various questions are raised during the discussions.

This paper is expanded from a chapter of the author's thesis \cite{Z}. A previous version was titled {\it Geometric Structures and Kodaira Dimensions}. We thank Tian-Jun Li for his great interest and constant encouragement for this work from a very early stage. The author is grateful to Claude LeBrun, Yi Liu, Pablo Su\'arez-Serrato, Rafael Torres and Yunhui Wu for their very useful comments. 

Without otherwise mentioning, our objects are connected closed oriented manifolds.

\section{Kodaira dimension of $3$-manifolds}
\subsection{Eight geometries and the definition}
We start with the discussion of the geometrization theorem, which says that every closed $3$-manifold can be decomposed uniquely into pieces that each has one of Thurston's eight geometric structures.  

The first step of the decomposition is the prime
decomposition, the existence and uniqueness is due to Kneser and Milnor respectively.

{\theorem \label{thurston} Every compact, orientable $3$-manifold
can be decomposed into the connected sum of a unique (finite)
collection of prime $3$-manifolds. }

This reduces the study further decomposition to
prime manifolds. The geometrization theorem was
conjectured by Thurston and finally proved by Perelman using Hamilton's Ricci flow \cite{P1, P2, P3} (see also \cite{CZ, KL, MT}).

{\theorem \label{geometrization}Every oriented prime closed $3$-manifold can be cut
along tori, so that the interior of each of the resulting
manifolds has a geometric structure with finite volume. }

Notice that there is a unique minimal way of cutting an irreducible oriented 3-manifold along tori into pieces that are Seifert manifolds or atoroidal called the JSJ decomposition. It is not quite the same as the decomposition in the above theorem, because some of the pieces in the JSJ decomposition might not have finite volume geometric structures. Moreover, there are many inequivalent cuttings of Theorem \ref{geometrization} depending on the initial metric to start the Ricci flow.

Let us digress on the definition of a geometry structure mentioned in Theorem \ref{geometrization}.
A {\it model geometry} is a simply connected smooth manifold $X$ together with a transitive action of a Lie group $G$ on $X$ with compact stabilizers.
A model geometry is called {\it maximal} if $G$ is maximal among groups acting smoothly and transitively on $X$ with compact stabilizers. Sometimes this condition is included in the definition of a model geometry.
A {\it geometric structure} on a manifold $M$ is a diffeomorphism from $M$ to $X/\Gamma$ for some model geometry $X$, where $\Gamma$ is a discrete subgroup of $G$ acting freely on $X$. If a given manifold admits a geometric structure, then it admits one whose model is maximal. In other words, different geometries are distinguished from their fundamental groups. In this paper, a {\it geometry} means a maximal model geometry such that at least one model $X/\Gamma$ has finite volume. 

There is a unique geometry  in dimension one,  that of the Euclidean line $\mathbb E^1$ (or sometimes denoted by $\mathbb R$). There are three geometries in dimension $2$: the spherical geometry $S^2$, the Euclidean geometry $\mathbb E^2$ and the hyperbolic geometry $\mathbb H^2$. 

In dimension three we have the following eight maximal
geometric structures:

\begin{enumerate}
    \item Spherical geometry $S^3$;
    \item The geometry of $S^2 \times \mathbb{E}$;
    \item Euclidean geometry $\mathbb{E}^3$;
    \item Nil geometry $Nil$;
    \item Sol geometry $Sol$;
    \item The geometry of $\mathbb{H}^2\times \mathbb{E}$;
    \item The geometry  $\widetilde{SL_2(\mathbb{R})}$; 
    \item Hyperbolic geometry $\mathbb{H}^3$.
\end{enumerate}

Here $\widetilde{SL_2(\mathbb{R})}$ is the universal cover of $PSL_2(\mathbb{R})$, the unit tangent bundle of $\mathbb H^2$. The spherical geometry $S^3$ could also be viewed as the double cover of the unit tangent bundle of $S^2$. The Nilpotent group is the group of $3\times 3$ upper triangular matrices of the form 
 \[B=\begin{pmatrix}
1 & b & c \\
0 & 1 & a \\
0 & 0 & 1 \end{pmatrix}\]
  The solvable group $Sol=\mathbb R^2\rtimes_{\phi} \mathbb R$, where $\phi(t)(x, y)=(e^tx, e^{-t}y)$.

To define the Kodaira dimension of $M^3$, we divide
the eight Thurston geometries into three categories:

$$\begin{array}{rl}
    -\infty: &\hbox{$S^3$ and $S^2 \times \mathbb{E}$}; \cr
     0: &\hbox{$\mathbb{E}^3$, $Nil$ and $Sol$};\cr
     1: &\hbox{$\mathbb{H}^2\times \mathbb{E}$, $\widetilde{SL_2(\mathbb{R})}$, $\mathbb H^3$}.
\end{array}$$

Given a $3$-manifold $M^3$,  we first decompose it into prime pieces
and then further exploit a  toroidal decomposition
for each prime summand, such that at the end each piece admits one of the eight geometric structures with finite volume. By Theorem \ref{thurston}, the decomposition is unique. We call this a $T$-decomposition. For example, 
$\mathbb RP^3\# \mathbb RP^3$ admits a geometric structure of type
 $S^2 \times \mathbb{R}$. But in this
paper, we should first decompose it into two $\mathbb RP^3$, then these
two prime pieces admit spherical geometry.

We are ready to give the following definition of Kodaira dimension of $3$-manifolds:

{\definition \label{t} For an oriented $3$-dimensional manifold $M^3$, we
define the Kodaira dimension $\kappa^t(M^3)$ as follows:

\begin{enumerate}
    \item $\kappa^t(M^3)=-\infty$ if for any $T$-decomposition, each piece has geometric type in category $-\infty$;
    \item $\kappa^t(M^3)=0$ if for any $T$-decomposition, we have at least a piece with geometry type in category $0$, but no piece has type in category $1$;
    \item $\kappa^t(M^3)=1$ if for any $T$-decomposition, we have at least one piece in category $1$.
\end{enumerate}
}

For non-orientable $M$, we define $\kappa^t(M)$ to be that of its oriented double cover $\widetilde M$.

It is worth noting that $\kappa^t$ is not automatically well-defined since the decomposition in Theorem \ref{geometrization}. In fact, depending on the choice of the initial metric, the Ricci flow will cut up a manifold into geometric pieces in many inequivalent ways. 

Before showing the Kodaira dimension $\kappa^t(M^3)$ is well-defined, let us recall an important result of Thurston (see Theorem 4.7.10 of \cite{Thur}), which is used several times throughout the paper.

\begin{theorem}[Thurston]\label{finiteclosed}
Non-closed $3$-dimensional geometric manifolds with finite volume exist only for geometries in category $1$, {\it i.e.} $\mathbb H^3$, $\mathbb H^2\times \mathbb E$ and $\widetilde{SL_2(\mathbb{R})}$.
\end{theorem}

We are ready to show that the Definition \ref{t} is well-defined.

{\theorem Definition \ref{t} is well-defined.}
\begin{proof}
For a manifold $M^3$, if we have a decomposition with every piece
 from category $-\infty$, then by Theorem \ref{finiteclosed}, there are no toroidal
decompositions. Because the prime decomposition is unique, then we
know that the case $\kappa^t=-\infty$ is well-defined.

Similarly, if we have a decomposition with at least one piece from
category $0$, but no piece from category $1$, then by Theorem \ref{finiteclosed},
there are no toroidal decompositions. Because the prime decomposition
is unique, then we know that the case $\kappa^t=0$ is well-defined.

Finally, the case $\kappa^t=1$  is just the complementary of cases
$\kappa^t=0$ and $\kappa^t=-\infty$. 

Hence, Definition \ref{t} is well-defined.
\end{proof}

Furthermore, we can classify the manifolds with $\kappa^t=-\infty$
or $0$. 

{\prop \label{infty}
Let $M^3$ be a $3$-dimensional manifold with $\kappa^t(M^3)=-\infty$. Then $M=(M_1\#
M_2\# \cdots \# M_n)$, where each $M_i$ is
prime and of the following types:
\begin{enumerate}
    \item spherical, i.e. it has a Riemannian metric of constant positive sectional
    curvature;
    \item $S^2\times S^1$;
    \item nontrivial $S^2$ bundle over $S^1$;
\end{enumerate}
}

\begin{proof}
We first decompose it into prime manifolds. Then we decompose
those prime ones to pieces of finite volume with only geometries
in category $-\infty$. If some geometric piece with finite volume has
spherical geometry $S^3$, then by elliptization conjecture which
is a corollary of Theorem \ref{geometrization}, it has a Riemannian
metric of constant positive sectional curvature. If some geometric piece with finite volume has the
geometry of $S^2 \times \mathbb{R}$, then it is one of the remaining
types listed in the statement of this theorem. In particular, these geometric pieces are all compact without boundary. It shows that
there are no further toroidal decompositions after prime decomposition. This
finishes the proof.
\end{proof}

For Kodaira dimension $0$ case, we have the following similar
result.

{\prop \label{0}Let $M^3$ be a $3$-dimensional manifold with $\kappa^t(M^3)=0$. Then $M=(M_1\# M_2\#
\cdots \# M_n)$, where each $M_i$ is prime and of
the following types:
\begin{enumerate}
    \item spherical, i.e. it has a Riemannian metric of constant positive sectional
    curvature;
    \item $S^2\times S^1$;
    \item nontrivial $S^2$ bundle over $S^1$;
    \item Seifert fibrations with zero orbifold Euler characteristic; 
    \item the mapping torus of an Anosov map of the 2-torus or quotient of these by groups of order at most 8.
\end{enumerate}

Moreover, at least one $M_i$ is of type $(4)$ or $(5)$.
}

\begin{proof}
The proof is similar to that of Proposition \ref{infty}.  First, the finite volume geometric pieces with structures
in category $0$ are also compact without boundary. Hence we still only
have prime decompositions. Second,  we have three more
geometries: $\mathbb E^3$, Nil, Sol. Euclidean and Nil are Seifert fiber spaces with orbifold Euler number $0$. 
Compact manifolds with Sol geometry are either  $T^2$
bundles over $S^1$ with monodromy of Anosov type or quotient of these by groups of order at most 8. These correspond to the last two types listed in the statement.
\end{proof}

\begin{remark}
There are two more non-orientable prime $3$-manifolds with the geometry in category $-\infty$ and $0$: $\mathbb RP^2\times S^1$ and the mapping torus of the antipode map of $S^2$, which is the non-orientable fiber bundle of $S^2$ over $S^1$. 
\end{remark}

\subsection{Nonzero degree maps}  In this section, we would like to discuss that how the Kodaira dimensions of $3$-manifolds change under non-zero degree maps.

When $f$ is a degree $k>2$ map between $M$ and $N$, it
can be deformed to a branched covering whose branch locus is a link. 
Hence the possible definition of
relative Kodaira dimension has its own interests for study, whence the branched locus is non-empty.

We start with showing $\mathbb H^3$ is the ``largest" geometry among the eight. It follows from basic properties of the  {\it Gromov norm}, or sometimes called the
{\it simplicial volume}. It is a norm on the homology (with real coefficients) given by minimizing the sum of the absolute values of the coefficients over all singular chains representing a cycle. The Gromov norm $||M||$ of the manifold $M$ is the Gromov norm of the fundamental class.  More precisely, let $|\cdot |_1: C_k(M; \mathbb R)\rightarrow \mathbb R$ be the $l^1$ norm on real singular chains: for $z=\sum c_i\sigma_i\in C_k(M; \mathbb R)$, $$|z|_1:=\sum |c_i|.$$ Then the Gromov norm is $$||M||:=\inf\{|z|_1|[z]=[M]\}\in \mathbb R_{\ge 0}.$$

We notice that the Gromov norm is additive when  gluing along tori. This fact implies that the
Gromov norm of a $3$-manifold is proportional to the sum of the
volume of the hyperbolic pieces under a geometric decomposition. In
particular, a $3$-manifold has zero Gromov norm if and only if this
is a graph manifold. 

\begin{lemma} \label{hyperbolic} Suppose $f:M^3\longrightarrow N^3$ is a non-zero degree map.
If $\kappa^t(N)=1$ and at least one of the geometric pieces has
geometry $\mathbb H^3$, then $\kappa^t(M)=1$. Moreover, at least one
of the geometric pieces for $M$ is hyperbolic.
\end{lemma}
\begin{proof}

Since $N$ has a hyperbolic piece, $||N||>0$. Then by the
definition of Gromov norm, $||M||\ge \deg(f)\cdot ||N||>0$. Indeed,
$\kappa^t(M)=1$ and $M$ also has a hyperbolic piece.
\end{proof}

By a result of Rong \cite{Rong}, we also know that a non-zero degree map between
Seifert manifolds with infinite $\pi_1$ is homotopic to a fiber
preserving pinch followed by a fiber preserving branched covering.
In this case, we can reduce our situation to that of dimension
$2$, and get $\kappa^t(M)\ge
\kappa^t(N)$ in turn.

The next theorem is the first step toward a definition of the relative Kodaira dimension of $3$-manifolds. We will need the following lemma (see for example Lemma 1.2 in \cite{Rong}).
\begin{lemma}\label{rolfsen}
If there is a non-zero degree map $f: M \rightarrow N$, then $f_*\pi_1(M)$ has finite index in $\pi_1(N)$.
\end{lemma}

\begin{theorem}\label{mapKod3}
If $f:M^3\longrightarrow N^3$ is a non-zero degree map, then $\kappa^t(M)\ge \kappa^t(N)$.
\end{theorem}
\begin{proof}
Since $3$-manifolds are almost determined by their fundamental groups, let us first recall that how  the fundamental groups determine geometric manifolds. In the following bullets, let $L$ be a geometric manifold.
\begin{itemize}
\item  $\pi_1(L)$ is finite if and only if the geometric structure on $L$ is spherical.
\item $\pi_1(L)$ is virtually cyclic but not finite if and only if the geometric structure on $L$ is $S^2\times \mathbb E$.
\item $\pi_1(L)$ is virtually abelian but not virtually cyclic if and only if the geometric structure on $L$ is Euclidean.
\item $\pi_1(L)$ is virtually nilpotent but not virtually abelian if and only if the geometric structure on $L$ is $Nil$.
\item $\pi_1(L)$ is virtually solvable but not virtually nilpotent if and only if the geometric structure on $L$ is $Sol$.
\item $\pi_1(L)$ has an infinite index normal cyclic subgroup but is not virtually solvable if and only if  the geometric structure on $L$ is either $\mathbb H^2 \times \mathbb E$ or $\widetilde{SL_2(\mathbb R)}$.
\end{itemize}

Here a group $G$ is said to be virtually having some property $P$ if there is a finite index subgroup $H$ of $G$ which has this property $P$. We remark that a virtually abelian/nilpotent/solvable group has an infinite index normal cyclic subgroup. 

Before continuing the proof, let us recall that a $3$-manifold has Gromov norm zero if and only if it is a graph manifold. 
When a graph manifold has $\kappa^t\le 0$, then each irreducible piece in the prime decomposition is (closed and) geometric.  
Especially, the fundamental group is the free product of several groups of types we listed above, because the fundamental group of the connected sum is the free product of the fundamental groups when dimension is greater than two. However, for a general $3$-manifold, its fundamental group is the free product with amalgamation along torus or trivial group by Theorem \ref{geometrization}.

Let us first prove that when $\kappa^t(M)=-\infty$, then $\kappa^t(N)=-\infty$ as well. Hence  $\pi_1(M)$ is the free product of several virtually cyclic groups $G_i$. 
We denote a cyclic subgroup in $G_i$ as $H_i$. Especially, any subgroup of the free product of several virtually cyclic groups like $\pi_1(M)$ cannot contain an infinite index normal cyclic subgroup. We prove it by contradiction. If there is such a cyclic group $C$ which is generated by an element from some $G_i$, then any element $a$ satisfying the property $a^{-1}Ca\subset C$ is contained in $G_i$. This is because otherwise there will be a nontrivial relation involving elements of $G_i$ and of at least another $G_j$, which contradicts to the fact that $\pi_1(M)$ is a free product of $G_i$. On the other hand, if $C$ is generated by a ``mixed" element that is not in a single $G_i$. Without loss, we could assume the generator of $C$ cannot be written as a power of another element. Then any element $a$ satisfying the property $a^{-1}Ca\subset C$ is contained in $C$. By Lemma \ref{rolfsen}, $f_*(\pi_1 (M))$ is of finite index in $\pi_1(N)$. Thus $f_*(\pi_1 (M))$ is the free product of cyclic groups and of finite index in $\pi_1(N)$. Because $f_*$ is a group homomorphism, any subgroup of $f_*(\pi_1 (M))$ also does not contain an infinite index normal cyclic subgroup. So if $\pi_1(N)$ contains a subgroup $G$ with an infinite index normal cyclic subgroup $C'$, then $C''=f_*(\pi_1(M))\cap C'$ is a normal subgroup of $H=f_*(\pi_1(M))\cap G\le f_*(\pi_1(M))$ with infinite index since $H/C''$ is of finite index in $G/C'$. Moreover, the group $C''$ is a nontrivial subgroup of $f_*(\pi_1(M))$ since otherwise $C'\cdot f_*(\pi_1(M))$ will be infinite distinct cosets. Hence $\pi_1(N)$ also does not contain a subgroup with an infinite index normal cyclic subgroup. In addition by Lemma \ref{hyperbolic}, $||N||=0$ and hence $N$ has to be a graph manifold. If a group has a subgroup with an infinite index normal cyclic subgroup, this subgroup will be preserved under free product or free product with amalgamation along tori. Thus $\kappa^t(N)=-\infty$. 

Similarly when $\kappa^t(M)=0$, $\pi_1(M)$ is the free product of  several virtually solvable groups. By the similar reasoning as above,  any subgroup of the free product of virtually solvable groups cannot contain an infinite index normal cyclic subgroup which is not virtually solvable. There are still two possibilities: when the cyclic group is contained in some (virtually solvable group) $G_i$ and when it is not. In the first case, the elements $a$ such that $a^{-1}Ca\subset C$ are contained in $G_i$. Hence any such subgroup containing an infinite index normal cyclic subgroup is virtually solvable. In the second case, these elements $a$ are exactly the centralizers of $C$ and constitute a cyclic group as well. 
Then by Lemma \ref{rolfsen}, $\pi_1(N)$ contains as a subgroup of finite index $f_*(\pi_1 (M))$, which is the free product of several virtually solvable groups as this property is preserved under group homomorphism. In addition, any subgroup of $f_*(\pi_1 (M))$ does not contain an infinite index normal cyclic subgroup which is not solvable. By the same argument as in the case $\kappa^t(M)=-\infty$, so is $\pi_1(N)$. And by Lemma \ref{hyperbolic}, $||N||=0$ and hence $N$ has to be a graph manifold. Thus $\kappa^t(N)\le 0$.

This completes our proof. 
\end{proof}

In the proof of the theorem, we see that non-zero degree map provides finer order of  geometric structures. 
Moreover, Lemma \ref{hyperbolic} shows that $\mathbb H^3$ is a ``larger" geometry than others. 
In general, we think that the Thurston norm and Gabai's result
on taut foliation should be useful for a version of Lemma \ref{hyperbolic} for other geometries. 

\subsection{Comparing with the Kodaira dimensions of $4$-manifolds}

\subsubsection{Additivity}\label{taubesconj}
We mentioned in the introduction that for complex manifolds,
sympletic $4$-manifolds and Lefschetz fibrations for $4$-manifolds, we also have
suitable definitions of Kodaira dimensions. In this section, we
will compare our Kodaira dimension $\kappa^t$ with these ones. 

Let us first recall the definitions.

\begin{definition} Suppose $(M, J)$ is a complex manifold of real dimension $2m$.
The holomorphic Kodaira dimension $\kappa^{h}(M,J)$ is defined as
follows:

\[
\kappa^{h}(M,J)=\left\{\begin{array}{cl}
-\infty &\hbox{ if $P_l(M,J)=0$ for all $l\ge 1$},\\
0& \hbox{ if $P_l(M,J)\in \{0,1\}$, but $\not\equiv 0$ for all $l\ge 1$},\\
k& \hbox{ if $P_l(M,J)\sim cl^k$; $c>0$}.\\
\end{array}\right.
\]

\end{definition}

Here  $P_l(M,J)$ is the $l$-th plurigenus of the  complex manifold
$(M, J)$ defined by $P_l(M,J)=h^0({\mathcal K}_J^{\otimes l})$, with
${\mathcal K}_J$ the canonical bundle of $(M,J)$.

\begin{definition}\label{sym Kod'}
For a minimal symplectic $4$-manifold $(M^4,\omega)$ with symplectic
canonical class $K_{\omega}$,   the Kodaira dimension of
$(M^4,\omega)$
 is defined in the following way:

$$
\kappa^s(M^4,\omega)=\begin{cases} \begin{array}{cll}
-\infty & \hbox{ if $K_{\omega}\cdot [\omega]<0$ or} & K_{\omega}\cdot K_{\omega}<0,\\
0& \hbox{ if $K_{\omega}\cdot [\omega]=0$ and} & K_{\omega}\cdot K_{\omega}=0,\\
1& \hbox{ if $K_{\omega}\cdot [\omega]> 0$ and} & K_{\omega}\cdot K_{\omega}=0,\\
2& \hbox{ if $K_{\omega}\cdot [\omega]>0$ and} & K_{\omega}\cdot K_{\omega}>0.\\
\end{array}
\end{cases}
$$

The Kodaira dimension of a non-minimal manifold is defined to be
that of any of its minimal models.
\end{definition}

Here $K_{\omega}$ is defined as the first Chern class of the
cotangent bundle for any almost complex structure compatible with
$\omega$. 

LeBrun \cite{LeBmrl, LeBcag, LeB} has studied the relations between the Yamabe invariant and Kodaira dimensions. Especially, he proposed a definition of general type for arbitrary $4$-manifolds in \cite{LeB}: namely, when the Yamabe invariant is negative. Recall that the Yamabe invariant is defined as $$Y(M)=\sup_{\mathcal [\hat g]\in \mathcal C} \inf_{g\in [\hat g]}\int_M s_{g}dV_{g},$$ where $g$ is a Riemannian metric on $M$, $s_g$ is the scalar curvature of $g$, and $\mathcal C$ is the set of conformal classes on $M$. When $Y(M)\le 0$, the invariant is simply the supremum of the scalar curvatures of unit-volume constant-scalar-curvature metrics on $M$. There is an interesting question of LeBrun: if $M^4$ admits a symplectic structure and $Y(M^4)<0$, is $\kappa^s(M^4)=2$? It is clear that $\kappa^s(M^4)=2$ would imply $Y(M^4)<0$ since $Y(M)\le -4\pi\sqrt{2K_M^2}$ when $M^4$ is minimal, the Seiberg-Witten invariant is nonzero and $K_M^2\ge 0$ (see {\it e.g.} \cite{LeBmrl}). On the other hand, the answer to LeBrun's question is positive for K\"ahler surfaces \cite{LeBcag}.

Finally, let us recall that the Kodaira dimension $\kappa^l(g,h,n)$ of Lefschetz fibrations defined in \cite{DZ}. Here $g$ and $h$ denote the fiber and base genus of a Lefschetz fibration and $n$ is the number of singular fibers. 

\begin{definition} \label{Lef Kod}
Given a relative minimal $(g,h,n)$ Lefschetz fibration with
$h\ge 1$, define the Kodaira dimension $\kappa^l(g,h,n)$ as follows:

\[
\kappa^l(g,h,n)=\begin{cases}\begin{array}{cll}
-\infty &\hbox{if $g=0$},\\
0& \hbox{if $(g,h,n)=(1,1,0)$},\\
1& \hbox{if $(g,h)=(1, \ge 2)$ or $(g,h,n)=(1,1,>0)$ or $(\ge 2,1,0)$},\\
2& \hbox{if $(g,h)\ge (2,2)$ or $(g,h,n)=(\ge 2,1,\ge 1)$}.
\end{array}
\end{cases}
\]

The Kodaira dimension of a non-minimal Lefschetz fibration with $h\ge 1$ is defined to be that of its minimal models.
\end{definition}

Here, a Lefschetz fibration is called relative minimal if no fiber contains a sphere of self-intersection $-1$.

Now we are ready to compare these Kodaira dimensions with our $\kappa^t$.

{\prop $\kappa^t(M)=\kappa^h(M^3\times S^1)$ when $M^3\times S^1$ admits a
complex structure. $\kappa^t(M)=\kappa^l(M^3\times S^1)$ when $M^3\times S^1$ admits
a Lefschetz fibration. $\kappa^t(M)=\kappa^s(M^3\times S^1)$ when $M^3\times S^1$ admits a
symplectic structure. In all these cases, the manifold $M$ is a surface
bundle over $S^1$. }

\begin{proof}
In \cite{E}, Etg\"{u} proved that when $M\times S^1$ admits a complex structure or a Lefschetz fibration, $M$
is a surface bundle over $S^1$. Then from the genus of
the surfaces, we determine the Kodaira dimension: when the
surface is $S^2$, $T^2$ or $\Sigma_g$ ($g\ge 2$)
respectively, $\kappa^t=-\infty,\ 0$ or $1$ respectively by
Definition \ref{t}. At the same time, $\kappa^l=-\infty, \ 0$ or
$1$ respectively by Definition \ref{Lef Kod}. Furthermore, when $M$
is a surface bundle over circle, $M^3\times S^1$ is a
surface bundle over torus. Thus by classification results on
these manifolds in \cite{DZ}, we also have $\kappa^t=\kappa^h$ in
this case.

When $M^3\times S^1$ admits a sympletic structure, then
$\kappa^t=\kappa^s$ is a consequence of the Taubes conjecture proved by Friedl and Vidussi \cite{FV}. Their theorem says that $M^3\times S^1$ admits a sympletic structure
if and only if $M^3$ is a surface bundle over circle. Then the
same argument as above shows that $\kappa^t=\kappa^s$.
\end{proof}

{\remark The corresponding results of \cite{FV} for circle bundles and mapping tori are further discussed in \cite{FV2} and \cite{LN} respectively. }

Finally, let us discuss more on additivity in the sense of \cite{LZadd} up to dimension $4$. Roughly speaking, we call the Kodaira dimension of a fibration is additive, if the Kodaira dimension of the total space is the sum of the Kodaira dimensions of the fiber and the base (might be in the relative sense if the fibration is not a bundle). For convenience of discussion, we could also define the topological Kodaira dimension $\kappa^t$ for manifolds of dimension up to $2$.

The $2$-dimensional Kodaira dimension is defined in the normal sense by the sign
of the Euler class. Namely, $\kappa^t(S^2)=-\infty$, $\kappa^t(T^2)=0$ and $\kappa^t(\Sigma_g)=1$ for $g\ge 2$. 

The only closed connected $0$-dimensional manifold is a point, and the only closed
connected $1$-dimensional manifold is diffeomorphic to a circle. We define $\kappa^t$ of them to be $0$.

Bundles in dimension three contains three cases: covering spaces, circle bundles over surface and surface bundles over circle. The covering map preserves $\kappa^t$ follows from the fundamental group description of geometric structures in Theorem \ref{mapKod3}. The additivity of other two cases are both straightforward to check by definition. 

For dimension four, discussions in this section imply the additivity for the product $M^3\times S^1$. For a surface bundle over surface, when the base is a positive genus surface, the additivity is established in \cite{DZ}. When the base
is $S^2$, the bundle is either a ruled surface or a Hopf surface;
the latter case occurs when the fiber is $T^2$ and homologically
trivial. Hence the additivity holds.

\subsubsection{Symplectic $4$-manifolds and contact type hypersurfaces}
Let $(M, \omega)$ be a symplectic $4$-manifold. A real hypersurface $V\subset M$ is a contact type hypersurface if $\omega$ can be written in a neighborhood of $V$ as $d\lambda$ for some $1$-form $\lambda$ whose restriction to $V$ is a contact form. An important family of examples consists of unit cotangent bundle of Lagrangian surfaces. Let us first review the embedded Lagrangian surfaces in symplectic $4$-manifolds.

In a projective surface $X(\mathbb C)$, connected components of real algebraic surface $X(\mathbb R)$ are Lagrangian surfaces. It is known that the topology of real algebraic surfaces are bounded by that of its corresponding complex algebraic surface. In general, let $X$ be a projective variety, then 
\begin{itemize}
\item (Thom \cite{Thom}): $\sum_ih^i(X(\mathbb R), \mathbb Z_2)\le \sum_ih^i(X(\mathbb C), \mathbb Z_2)$,
\item (Sullivan \cite{Sull}): $\chi(X(\mathbb R))\equiv \chi(X(\mathbb C))\quad \mod \  2$.
\end{itemize}

It is also known as Comessatti's theorem that when $X$ is smooth, rational, projective surface, and $X(\mathbb R)$ is orientable, then $X(\mathbb R)$ is either a sphere or a torus. Especially, rational $X(\mathbb R)$ are all connected.

For Lagrangian surfaces in a symplectic $4$-manifold, in general we do not have Thom's or Sullivan's results, except when the Lagrangian $L$ is the fixed loci of an anti-symplectic involution $\iota$ with $\iota^*\omega=-\omega$.

 However, we have a generalization of Comessatti's result. The following is well known. We include it for completeness.

\begin{prop}
Let $(M, \omega)$ be a symplectic $4$-manifold with $b^+(M)=1$. If $L\subset (M, \omega)$ is an orientable Lagrangian, then $L$ is either a sphere or a torus. When $L$ is a torus, it is null-homologous.
\end{prop}
\begin{proof}
According to Gompf \cite{Gom}, we could perturb $\omega$ to a nearby new symplectic form $\Omega$ such that $L$ is a symplectic surface of $(M, \Omega)$ when $L$ is orientable. Now if $g(L)\ge 1$, then by Weinstein neighborhood theorem, the self-intersection number $\hbox{PD}[L]\cdot \hbox{PD}[L]=2g(L)-2\ge 0$. We are now in the situation to apply the light cone lemma: $\hbox{PD}[L]\cdot \hbox{PD}[L]\ge 0$, $[\omega]^2>0$, $\hbox{PD}[L]\cdot [\omega]=0$ and  $\hbox{PD}[L]\cdot [\Omega]>0$, $[\omega]\cdot [\Omega]>0$. Hence, $[L]=0$ and $L$ is a null-homologous torus.
\end{proof}

This especially tells us that all orientable Lagrangians in rational and ruled surfaces are spheres and tori. However, it also tells us that one cannot distinguish rational and ruled surfaces from other symplectic $4$-manifolds of $b^+=1$ by the topology of their Lagrangian submanifolds. Therefore we ask: can we distinguish them by contact type hypersurfaces? 
\begin{question}
If $V\subset (M, \omega)$ is a contact type hypersurface in a rational or ruled symplectic $4$-manifold $M$, do we have $\kappa^t(V)\le 0$? Especially, can $V$ be hyperbolic?
\end{question}

It is shown in \cite{Wel} that the unit cotangent bundle of an orientable hyperbolic Lagrangian surface (these are of geometry $\widetilde{SL_2(\mathbb R)}$) does not embed as a hypersurface of contact type in a rational or ruled symplectic $4$-manifold.

When $M=T^4$ with standard symplectic structure $dX_1\wedge dX_2+dX_3\wedge dX_4$, we also want to know whether a contact type hypersurface could be hyperbolic or not. However, on the other hand, any orientable surface can be realized as Lagrangian submanifold of it. First we have two transversal families of Lagrangian tori $\Gamma_{13}=\{X_1\times \{x_2\}\times X_3 \times \{x_4\}\}$ and $\Gamma_{24}=\{\{x_1\}\times X_2\times \{x_3\} \times X_4\}$. Members from $\Gamma_{13}$ and $\Gamma_{24}$ transversally intersect at one point. Use Lagrangian surgery \cite{Pol} to resolve the intersection point, we will have a Lagrangian genus two surface. Similarly, if we resolve the intersection points of the configuration with $g-1$ different members from $\Gamma_{13}$ and one member from $\Gamma_{24}$, we would have a Lagrangian surface of genus $g$.

\section{Kodaira dimensions and geometric structures of $4$-manifolds}

\subsection{Geometries in dimension four} As we have seen, there are Kodaira dimensions available for complex and symplectic $4$-manifolds. On the other hand, we also have $19$ geometries in dimension $4$ (see \cite{Filip}). It is natural to ask whether we could define Kodaira dimension for $4$-manifolds, at least for irreducible ones, through the $19$ geometries.
In the following, we will briefly discuss the relation between the $19$ geometries in dimension $4$ and the Kodaira dimension. In dimension $4$, we do not have the decomposition theorem as in dimension $3$. Hence the discussion here is only about closed manifolds with one of the $19$ geometries.

First, we separate the $19$ geometries into $4$ categories: 
$$\begin{array}{rl}
    -\infty: &\hbox{$\mathbb P^2(\mathbb C)$, $S^4$, $S^3 \times \mathbb{E}$, $S^2 \times S^2$, $S^2 \times \mathbb{E}^2$, $S^2 \times \mathbb{H}^2$, $Sol_0^4$ and $Sol_1^4$ };\cr
     0: &\hbox{$\mathbb{E}^4$, $Nil^4$, $Nil^3 \times \mathbb E$ and $Sol_{m,n}^4$(including $Sol^3\times \mathbb E$)};\cr
     1: &\hbox{$\mathbb H^2 \times \mathbb E^2$, $\widetilde{SL_2} \times \mathbb E$, $\mathbb H^3 \times \mathbb E$ and $F^4$};\cr
     2: &\hbox{$\mathbb{H}^2(\mathbb C)$, $\mathbb H^2 \times \mathbb H^2$ and $\mathbb{H}^4$}.
\end{array}$$

Let us recall the definition of non-product geometries in the list. First, $S^4$, $\mathbb H^4$, $\mathbb P^2(\mathbb C)$ and $\mathbb H^2(\mathbb C)=SU(2, 1)/S(U(2)\times U(1))$ are Riemannian symmetric spaces. 

Next, nilpotent Lie groups and solvable Lie groups are realized by semidirect product: $Nil^4=\mathbb R^3\rtimes_U \mathbb R$, $Sol^4_{m, n}=\mathbb R^3\rtimes_{T_{m,n}}\mathbb R$. Here $U(t)=\exp(tB)$ and $T_{m, n}(t)=\exp(tC_{m, n})$ with \[B=\begin{pmatrix}
0 & 1 & 0 \\
0 & 0 & 1 \\
0 & 0 & 0 \end{pmatrix}, \ \ \  C_{m,n}= \begin{pmatrix}
a & 0 & 0 \\
0 & b & 0 \\
0 & 0 & c \end{pmatrix}\]
where $e^a>e^b> e^c$ are roots of $\lambda^3-m\lambda^2+n\lambda-1=0$ with $m, n$ positive integers. Especially, $a>b>c$ are real and $a+b+c=0$. If $m=n$, then $b=0$ and $Sol^4_{m,n}=Sol^3\times \mathbb E$. 

When there are two equal roots for $\lambda^3-m\lambda^2+n\lambda-1=0$, i.e. when $m^2n^2+18mn=4(m^3+n^3)+27$, the geometry is denoted by $Sol_0^4$. There is another solvable group $Sol_1^4$ which is represented as a matrix group 
 \[B=\begin{pmatrix}
1 & \beta & \gamma \\
0 & a & \alpha \\
0 & 0 & 1 \end{pmatrix}, \ \ \  a, \alpha, \beta, \gamma\in \mathbb R, a>0\]

Finally we have the geometry $F^4$ with isometry group $\mathbb R^2\rtimes SL(2, \mathbb R)$ with the natural action of $SL(2, \mathbb R)$ on $\mathbb R^2$. The geometry $F^4$ is the only geometry in the list to admit no compact model (although, by definition of geometry, we have some models with finite volume). However, it does admit complex structures and even K\"ahler structures.

In \cite{W}, Wall studies the relations of complex structures and the geometries. In summary, the geometries
$$S^4, \mathbb H^4, \mathbb H^3 \times \mathbb E, Nil^4, Sol_{m,n}^4$$
do not admit a complex structure compatible with the geometric structure. In the remaining cases except $Sol_1^4$, the complex structure on the maximal relevant geometry is unique. For $Sol^4_1$ we have two complex structures, denoted by $Sol^4_1$ and $Sol'^4_1$. Among those admit complex structures, the geometries
$$S^3 \times \mathbb E, Sol_0^4, Sol_1^4, Sol'^4_1, Nil^3 \times \mathbb E, \widetilde{SL_2} \times \mathbb E$$
admit no compatible K\"ahler structures. The first four are in Class VII of Kodaira's list of complex surfaces. The rest two are in Class VI. A compact model for $Nil^3 \times \mathbb E$ is so-called Kodaira-Thurston manifold.  All the remaining geometries admit compatible K\"ahler structures. Moreover, the Kodaira dimension of these K\"ahler structures is the same as the category number of the corresponding geometric structures. 

Let us collect some useful information for the hyperbolic geometry $\mathbb H^3$ extracted from \cite{Scott}. We use the upper half $3$-space model $$\mathbb R^3_+=\{(x, y, z)\in \mathbb R^3| z>0\}$$ with the metric $$ds^2=\frac{1}{z^2}(dx^2+dy^2+dz^2).$$ The isometry group is generated by reflections and an isometry is determined by its restriction to  to $2$-sphere at infinity $\mathbb C\cup \{\infty\}$, where the $xy$-plane is identified with $\mathbb C$.

The group of orientation preserving isometries of $\mathbb H^3$ can be identified with the group of M\"obius transformations $PSL(2, \mathbb C)$ of $\mathbb C\cup \{\infty\}$. If we identify the point $(x, y, z)\in \mathbb R^3_+$ with the quaternion $x+yi+zj$. The $2\times 2$ complex matrix $\begin{pmatrix}
a & b  \\
c & d  \end{pmatrix}$ acts on $\mathbb R^3_+$ by $$w\mapsto (aw+b)(cw+d)^{-1},$$where $w$ is a quaternion of the form $x+yi+zj, z>0$. This yields all orientation preserving isometries of $\mathbb H^3$. It follows that each orientation preserving isometry of $\mathbb H^3$ fixes one or two points of the sphere at infinity. There isometries are call {\it parabolic} and {\it hyperbolic} respectively.  If $\alpha$ is an isometry of $\mathbb H^3$, let $\hbox{fix}(\alpha)$ denote the set of points on the sphere at infinity which are fixed by $\alpha$.

\begin{lemma}[Lemma 4.5 in \cite{Scott}]\label{4.5scott}
\begin{enumerate}
\item If $\alpha$ and $\beta$ are two non-trivial orientation preserving isometries of $\mathbb H^3$, then $\alpha$ and $\beta$ commute if and only if $\hbox{fix}(\alpha)=\hbox{fix}(\beta)$. 

\item If $\alpha$ is a non-trivial orientation preserving isometry of $\mathbb H^3$, then the group $C(\alpha)$ of all orientation preserving isometries which commute with $\alpha$ is abelian and isomorphic to $\mathbb R^2$ or $S^1\times \mathbb R$.
\end{enumerate}
\end{lemma}

We remark the orientation preserving isometry  group of $\mathbb H^2$ can be identified with $PSL(2, \mathbb R)$. There are three types of orientation preserving isometry: rotations, parabolics and hyperbolics. They are characterized by the number of points, {\it i.e.} $0$, $1$ or $2$, left fixed on the circle at infinity.  Thus similar results of Lemma \ref{4.5scott} hold. For the second statement, $C(\alpha)$ is abelian and is isomorphic to $S^1$ if $\alpha$ is a rotation and isomorphic to $\mathbb R$ otherwise.

We have the following
\begin{prop}\label{H3R}
Let $G$ be a discrete group of isometries of $\mathbb H^3\times \mathbb E$ which acts freely and has quotient $M$. Then one of the following three statements holds:
\begin{enumerate}
\item the natural foliation of $\mathbb H^3\times \mathbb E$ by lines descends to an $S^1$ action on $M$;
\item the natural foliation of $\mathbb H^3\times \mathbb E$ by lines gives $M$ the structure of a line bundle over some hyperbolic $3$-manifold;
\item the natural foliation of $\mathbb H^3\times \mathbb E$ by lines descends to a foliation of $M$ by lines in which each line has non-closed image in $M$. In this case, $G$ must by isomorphic to $\mathbb Z$, $\mathbb Z\times \mathbb Z$ or the Klein bottle group. 
\end{enumerate}
 Especially, in the last two cases $M$ is not a closed manifold.
\end{prop}
\begin{proof}
We identify the isometry group of $\mathbb H^3\times \mathbb E$ with $\hbox{Isom}(\mathbb H^3)\times \hbox{Isom}(\mathbb R)$. As $G$ is discrete, $K=G\cap \hbox{Isom}(\mathbb R)$ is discrete and so must be $1$, $\mathbb Z_2$, $\mathbb Z$ or $D(\infty)$. As $G$ acts freely, $K$ is $1$ or $\mathbb Z$. Let $\Gamma$ denote the image of the projection $G\rightarrow \hbox{Isom}(\mathbb H^3)$. Then we have the exact sequence $$0\rightarrow K\rightarrow G\rightarrow \Gamma\rightarrow 0.$$ In the case when $K$ is infinite cyclic, each line $\{x\}\times \mathbb E$ descends to a circle. Hence the quotient manifolds would admit $S^1$ actions, or more precisely Seifert bundle structures, which implies the quotient manifolds have Gromov norm $0$. When $K$ is trivial, then $G\cong \Gamma$. If $\Gamma$ is a discrete group of isometries of $\mathbb H^3$, then the quotient would be a line bundle over $ \mathbb H^3/\Gamma$. For this case, the quotient is not closed. 

When $\Gamma$ is an indiscrete group of $\hbox{Isom}(\mathbb H^3)$.
Replacing $G$ by a subgroup of index two if necessary, we can suppose that $\Gamma$ is orientation preserving. We will now consider the projection $G\rightarrow \hbox{Isom}(\mathbb E)$. Let $L$ be the image of the kernel under the isomorphism $G\rightarrow \Gamma$. $L$ is a discrete group of isometries of $\mathbb H^3$. 

Suppose $L$ is non-trivial. Conjugation of $L$ by each element of $\Gamma$ induces an automorphism of $L$. As $L$ is discrete, an element of $\Gamma$ sufficiently close to the identity must commute with $L$. Since $\Gamma$ is not discrete, there must be a non-trivial element of $\Gamma$ which centralizes $L$. Now the centralizer of any non-trivial element in $\hbox{Isom}^+(\mathbb H^3)=PSL(2, \mathbb C)$ is always abelian, and actually $\mathbb R\times \mathbb R$ or $S^1\times \mathbb R$, by Lemma \ref{4.5scott}. So it follows $L$ is abelian. As $L$ is discrete and torsion free, $L$ must be $\mathbb Z$ or $\mathbb Z\times \mathbb Z$. On the other hand, $\Gamma$ is indiscrete, hence $\Gamma$ has a subgroup $\Gamma_1$ of index at most two which centralizes $L$. Since $L$ is abelian, each element has the same fixed points set by Lemma \ref{4.5scott}. Again by the same lemma, each element in $\Gamma_1$ has the same fixed points set and thus $\Gamma_1$ is abelian. If $\Gamma_1$ consists of hyperbolic isometries there is a unique geodesic $l$ in $\mathbb H^3$ left invariant by $\Gamma_1$. Let $G_1$ be the subgroup of $G$ corresponding to $\Gamma_1$. We see that $G_1$ leaves invariant the plane $l\times \mathbb R$. As this plane is isometric to Euclidean plane and $G_1$ must act discretely on it, we know $G_1$ is $\mathbb Z$ or $\mathbb Z\times \mathbb Z$. 

If $\Gamma_1$ consists of parabolic isometries, without loss we could assume the common fixed point at infinite $S^2$ is $\infty$. Hence $\Gamma_1$ leaves invariant each line $x=y=\hbox{const}$. Taking $l$ to be one of these lines, same argument as the above hyperbolic isometry case applies to complete the proof when $L$ is non-trivial.

When $L$ is trivial, since the orientation preserving part of  $\hbox{Isom}(\mathbb R)$ is isomorphic to $\mathbb R$. Hence $G$ has a subgroup of index at most two which is abelian. All the arguments above apply to show $G_1$ is actually $\mathbb Z$ or $\mathbb Z\times \mathbb Z$.  Especially, all these analysis imply that the quotient is not a closed manifold if not fibered by $S^1$.
\end{proof}

We have the following characterization of geometric closed $4$-manifolds in the top category by Gromov norm.

\begin{theorem}
A closed geometric $4$-manifold has nonzero Gromov norm if and only if it is in category $2$.

\end{theorem}
\begin{proof}First we show that for any geometric manifold in category $2$, the Gromov norm is non-vanishing. This is because if $M$ is a closed oriented locally symmetric space 
of non-compact type, then $||M||>0$ \cite{LS}. The examples include the three geometries in category $2$. Especially, if $M$ is hyperbolic $$||M||=\frac{1}{v_4}\hbox{Vol}(M),$$ where $v_4$ is the maximal volume of ideal geodesic triangles. If $M$ has geometry $\mathbb H^2\times \mathbb H^2$ \cite{BK}, $$||M||=\frac{3}{2\pi^2}\hbox{Vol}(M).$$

We then want to show that the Gromov norm vanishes for geometric $4$-manifolds in the other three categories. Closed manifolds of geometries $\mathbb E^4$, $Nil^4$, $Nil^3 \times \mathbb E$, $Sol_0^4$, $Sol_1^4$ and $Sol_{m,n}^4$ would have solvable fundamental group. The Gromov norm of such a closed manifold is zero since a solvable group is in particularly amenable. For the geometric closed manifolds in category $-\infty$, if the geometry is $\mathbb P^2(\mathbb C)$, $S^4$ or $S^2\times S^2$, then the fundamental group is finite and thus amenable. This implies the Gromov norm is zero. All the rest in category $-\infty$ have a factor of $S^2$ or $S^3$. Let $G$ be the discrete group of isometries of these geometries, then the original natural foliation by $2$-spheres or $3$-spheres of each geometry is preserved by the isometries. Hence any such geometric manifold would inherit a folation by $2$ or $3$ dimensional spherical geometries.
Since spherical geometries in dimension $2$ or $3$ admit nontrivial group actions, especially $S^1$ actions, this action would extend to the whole geometric manifold. Hence by \cite{Yano}, the Gromov norm is zero.

For geometries in category $1$, since $F^4$ does not admit any closed manifold model, we will focus on the rest three. For $\mathbb H^3\times \mathbb E$, the statement follows from Proposition \ref{H3R}. We have two more cases.
\smallskip
\smallskip

\noindent{\bf 1. $\mathbb H^2\times \mathbb E^2$:}

\smallskip
\smallskip
For $\mathbb H^2\times \mathbb E^2$, its isometry group is identified with $\hbox{Isom}(\mathbb H^2)\times \hbox{Isom}(\mathbb R^2)$. Hence $K=G\cap \hbox{Isom}(\mathbb R^2)$ is discrete and torsion free, so has to be $1$, $\mathbb Z$ or $\mathbb Z\times \mathbb Z$. If $K$  is non-trivial, then at least one line in $E^2$ descends to a circle. Hence the quotient manifolds admit an $S^1$ action, which implies the Gromov norm is zero. 

If $K=1$, we know again $G\cong \Gamma$ where $\Gamma$ is the image of the projection $G\rightarrow \hbox{Isom}(\mathbb H^2)$. If $\Gamma$ is discrete, $G$ cannot be cofinite since $$\hbox{vol}(\mathbb H^2\times \mathbb E^2/G)=\hbox{area}(\mathbb H^2/\Gamma)\cdot \hbox{area}(\mathbb E^2/K).$$ In fact, $M$ is a $2$-dimensional vector bundle over the hyperbolic surface $\mathbb H^2/\Gamma$. If $\Gamma$ is not discrete, we will show that $G$ has a subgroup of finite index isomorphic to $\mathbb Z$, $\mathbb Z\times \mathbb Z$ or $\mathbb Z\times \mathbb Z\times \mathbb Z$. The argument is similar to that of Proposition \ref{H3R}.  Let $L$ be the image of the kernel of the projection $G\rightarrow \hbox{Isom}(\mathbb E^2)$ under the isomorphism $G\rightarrow \Gamma$.  $L$ is a discrete group of isometries of $\mathbb H^2$. Suppose $L$ is non-trivial first. As $\Gamma$ is indiscrete, there must be a non-trivial element of $\Gamma$ which centralizes $L$. Now the centralizer of any non-trivial element in $PSL(2, \mathbb R)$ is always abelian, hence $L$ is $\mathbb Z$ since it is discrete and torsion free as well. Hence $\Gamma$ has a subgroup $\Gamma_1$ of index at most two which centralizes $L$, and $\Gamma_1$ is abelian.  By the remark after Lemma \ref{4.5scott} for $\hbox{Isom}(\mathbb H^2)$, $\Gamma_1$ consists of hyperbolic or parabolic isometries since $L\cong \mathbb Z$ prevents the case of rotations.
Moreover, whenever $\Gamma_1$ consists of hyperbolic isometries or parabolic isometries, there will be an Euclidean space $l\times \mathbb E^2\subset \mathbb H^2\times \mathbb E^2$ left invariant by $G_1$ and $G_1$ must act discretely on it. It follows that $G_1$ is $\mathbb Z$, $\mathbb Z\times \mathbb Z$ or $\mathbb Z\times \mathbb Z\times \mathbb Z$. 

If $L$ is trivial, then $G$ is identical to both $\Gamma$ and its image under the projection $G\rightarrow \hbox{Isom}(\mathbb E^2)$, say $F$. 
If $F$ is discrete, then it is clear that $M$ is an $\mathbb H^2$ bundle over the Euclidean surface $\mathbb E^2/F$.

Hence we could assume both $F$ and $\Gamma$ are not discrete. Recall the exact sequence $$0\rightarrow \mathbb R^2\rightarrow \hbox{Isom}(\mathbb E^2)\rightarrow O(2)\rightarrow 0.$$ 
Then $J=F\cap \mathbb R^2$ is a normal subgroup of $F$.

Since translations with same length are in the same conjugacy class, if $J$ is not discrete, we could choose a discrete subgroup $J'$ which is also normal in $F$. Since $J'$ is discrete and $F$ is indiscrete, $F$ has a subgroup $F_1$ of index at most two which centralizes $J'$. Since the rotations and translations do not commute, $F_1\cong F_1\cap \mathbb R^2$. In particular, this implies $F_1$ and its corresponding group $\Gamma_1$ in $\Gamma$ are abelian. There is an Euclidean plane (for rotations) or an Euclidean $3$-space (for hyperbolic or parabolic isometries) fixed by $G_1$ which acts discrete on it.
 Hence $G_1$ is $\mathbb Z$, $\mathbb Z\times \mathbb Z$ or $\mathbb Z\times \mathbb Z\times \mathbb Z$. This finishes the proof that closed $\mathbb H^2\times \mathbb E^2$ manifolds admit an $S^1$ action, and thus have vanishing Gromov norm.

\smallskip
\smallskip

\noindent{\bf 2. $\widetilde{SL_2} \times \mathbb E$:}

\smallskip
\smallskip
For $\widetilde{SL_2} \times \mathbb E$ geometry, first notice $\hbox{Isom}(\widetilde{SL_2} \times \mathbb E)=\hbox{Isom}(\widetilde{SL_2})\times \hbox{Isom}(\mathbb E)$. Let us look at the image of $G$ under the projection $\hbox{Isom}(\widetilde{SL_2} \times \mathbb E)\rightarrow \hbox{Isom}(\mathbb E)$. If the kernel $K=G\cap \hbox{Isom}(\widetilde{SL_2})$ is trivial, then the image has to be indiscrete. In this case, $G_1$ is an abelian group. Then we look at the other projection $\hbox{Isom}(\widetilde{SL_2} \times \mathbb E)\rightarrow \hbox{Isom}(\widetilde{SL_2})$. If the kernel is non-trivial, the quotient will be an $S^1$ manifold. Then $G$ is identified with its image under this projection. We further project it under $\hbox{Isom}(\widetilde{SL_2})\rightarrow \hbox{Isom}(\mathbb H^2)$. Again, if the kernel is nontrivial, it is an $S^1$ manifold. Hence we could assume $G$ is identified with its image under the composition of the above two projections. An abelian subgroup of $PSL(2, \mathbb R)= \hbox{Isom}(\mathbb H^2)$ fixes a line $l$. Hence $G_1$ leaves invariant and acts discretely on the $3$-space $l\times \mathbb E^2$. It follows that  
$G_1$ is $\mathbb Z$, $\mathbb Z\times \mathbb Z$ or $\mathbb Z\times \mathbb Z\times \mathbb Z$. 

If the kernel $K$ is nontrivial, it is discrete in $\hbox{Isom}(\widetilde{SL_2})$. There are three cases by the classification of $\widetilde{SL_2}$ geometry.  In the first case, the corresponding $3$-manifold is a line bundle over a non-closed surface. In this scenario, the line bundle structure would be inherited by the $4$-manifold. Hence, the quotient is a non-closed manifold. In the second case, the $3$-manifold is a Seifert fibration. In this situation, the quotient $4$-manifold would also be $S^1$ fibered and thus has Gromov norm zero.

We are left with the case when $K\cap \mathbb E$ is trivial and the image of $K$ under $\hbox{Isom}(\widetilde{SL_2})\rightarrow \hbox{Isom}(\mathbb H^2)$ is indiscrete. In this case, $K$ is $\mathbb Z$, $\mathbb Z\times \mathbb Z$ or the Klein bottle group. Especially, we notice that $G$ has a subgroup $G_1$ of index at most two such that it has discrete abelian normal subgroup $H$ which is isomorphic to $\mathbb Z$ or $\mathbb Z\times \mathbb Z$.

Then let us look at the image of $G$ under the composition of the projections $$\hbox{Isom}(\widetilde{SL_2}\times \mathbb E)\rightarrow \hbox{Isom}(\widetilde{SL_2})\rightarrow \hbox{Isom}(\mathbb H^2).$$ If the kernel is nontrivial, then the quotient $4$-manifold is $S^1$-fibered since the kernel in each step has to be infinite cyclic. Hence we could assume $G$ is identified with its image $\Gamma$ in $PSL(2, \mathbb R)=\hbox{Isom}(\mathbb H^2)$. Since $\Gamma$ is indiscrete and it has nontrivial discrete abelian normal subgroup $H$, $\Gamma$ has a subgroup $\Gamma_1$ of index at most two which centralizes $H$. Hence by Lemma \ref{4.5scott}, $\Gamma_1$ is abelian. Then as we argued in Proposition \ref{H3R}, $G_1$ leaves invariant a $3$-space $l\times \mathbb E^2$. Since this $3$-space is isometric to the Euclidean space and $G_1$ acts discretely on it, it follows $G_1$ is $\mathbb Z$, $\mathbb Z\times \mathbb Z$ or $\mathbb Z\times \mathbb Z\times \mathbb Z$. 

In summary, either our quotient manifold is not closed or it admits nontrivial $S^1$ action which implies Gromov norm $0$.
\end{proof}

\begin{remark}
In the above proof, we obtain more precise description of $4$-manifolds of geometries $\mathbb H^3\times \mathbb E$, $\mathbb H^2\times \mathbb E^2$, and $\widetilde{SL_2} \times \mathbb E$. Especially, we have shown that all such closed $4$-manifolds admit non-trivial $S^1$ action.
\end{remark}

We end up this section by discussing the relations with symplectic structures. First, symplectic $4$-manifolds with $\kappa^s=-\infty$ are rational or ruled surfaces. Hence, $\mathbb P^2(\mathbb C)$, $S^2 \times S^2$, $S^2 \times \mathbb{E}^2$ and $S^2 \times \mathbb{H}^2$  in category $-\infty$ admit symplectic models. Meanwhile, $S^4$, $Sol_0^4$ and $Sol_1^4$ and $S^3\times \mathbb E$ does not admit any symplectic models. Geometries $\mathbb E^4$, $Nil^4$, $Nil^3\times \mathbb E$ and $Sol^3\times \mathbb E$ in category $0$ admit symplectic models. They are realized by $T^2$ bundles over $T^2$ \cite{Gei}. All geometries in category $1$ except $F^4$ also admit symplectic structures. They are realized by surface bundles over torus. Finally, for geometries in category $2$, product of surfaces $\Sigma_g\times \Sigma_h$ has geometry $\mathbb H^2\times \mathbb H^2$, ball quotients have $\mathbb H^2(\mathbb C)$. However, it is conjectured that any closed hyperbolic $4$-manifold would have all Seiberg-Witten invariants vanish, which in particular implies it does not admit symplectic structures. Partial results towards this conjecture were obtained in \cite{Lebhyp}.

As suggested by the discussions in this section, we have the following question.

\begin{question}\label{topKod0norm}
\begin{enumerate}
\item Let $M$ be a smooth $2n$-dimensional complex manifold with nonvanishing Gromov norm. Is $\kappa^h(M)=n$?

\item Let $M$ be a smooth $4$-dimensional symplectic manifold with nonvanishing Gromov norm. Is $\kappa^s(M)=2$?
\end{enumerate}
\end{question}

If $M$ is a K\"ahler surface, the question is positively answered by the result of \cite{PP, PPinv}. Precisely, they showed that $M$ admits an $\mathcal F$-structure if and only if the Kodaira dimension is different from $2$ in \cite{PP}. On the other hand, the existence of $\mathcal F$-structure implies vanishing Gromov norm. Moreover, all known examples of compact complex surfaces which are not of K\"ahler type have $\mathcal F$-structure and thus vanishing Gromov norm. In other words, the complex part of Question \ref{topKod0norm} for complex surfaces is reduced to answer the following: Does every complex surface of Class VII have Gromov norm $0$? 

In general, if the answer is positive for $M$ and $N$, so is the product manifold $M\times N$ since the Kodaira dimension is additive and $$||M||\cdot ||N||\le ||M\times N||\le {\dim M+\dim N \choose \dim M}\cdot ||M||\cdot ||N||.$$

For the symplectic part, $\kappa^s(M)=-\infty$ implies vanishing Gromov norm since all these manifolds are rational or ruled surfaces which have amenable fundamental groups. It is most interesting to know whether $\kappa^s(M)=0$ would imply $||M||=0$. 
 

\subsection{Partial orders defined by non-zero degree maps}
In dimension two, orientable surfaces are ordered by their genus, which is finer than Kodaira dimension. 
This order could also be viewed as introduced by maps between manifolds. More precisely, there is a non-zero 
degree map from $\Sigma_g$ to $\Sigma_h$ if and only if $g\ge h$. 
Thus, we could introduce a partial ordered set (in this case it is totally 
ordered). The elements are surfaces up to homotopy/homeomorphism/diffeomorphism. And we say $M^2$ is larger than $N^2$, or $M^2\succ N^2$, 
if there is a non-zero degree map from $M^2$ to $N^2$. Notice that it does define an order because once $M^2\succ N^2$,
and $N^2\succ M^2$, then $M^2\cong N^2$.

This partial order could be generalized to higher dimensions. It is usually called the Gromov partial order. However, there are several issues. Let us 
first focus on closed orientable manifolds. First of all, 
it is very sensitive to the category of maps we choose. We are interested in continuous or differentiable maps, 
and sometimes may require to preserve the symplectic/complex structures. This is not a problem when the dimension is three. In dimension four, Duan and Wang \cite{DW} show that, when we are working
on continuous maps and topological manifolds, the simply connected $4$-manifolds are ordered by their intersection forms. There are topological $4$-manifolds  admit no smooth structures. However, if we concentrate on smooth manifolds, then it does not really matter if we look at continuous maps or differentiable map. This is because a continuous map between smooth manifolds is homotopic to a differentiable one ({\it c.f.} \cite{BT} Proposition 17.8). At the same time, degree is a homotopy invariant. In other words, the smooth non-zero degree map cannot distinguish the exotic smooth structures. For example, the exotic $7$-spheres and the standard $7$-sphere (smoothly) $1$-dominate each other, since we could find differentiable degree $1$ maps (from either direction) which is homotopic to homeomorphisms.

The second issue is this mapping ``order" is not necessarily a partial order. In other words, if we define that $M$ and $N$ are in the same equivalence class when we have a non-zero degree map from $M$ to $N$ and
a non-zero degree map from $N$ to $M$,  then the equivalence classes are no longer manifolds up to homotopy/homeomorphism/diffeomorphism, 
even in dimension three. For example, $S^3$ and the lens spaces $L(p,q)$ are in the same equivalence class 
because there is a quotient map from $S^3$ to $L(p,q)$, and we know that any three manifold dominates $S^3$. We are interested in determining the manifolds in a given class. Lemma \ref{rolfsen} is useful to deal with this issue.

The last but not the least,  we have to restrict ourselves on irreducible manifolds at least when dimension is four. Let us
take a look at an example of LeBrun \cite{LeBcover}. Suppose $M$ is a (non-spin) complete intersection surface of general type,
and $N=S^2\times S^2/\mathbb Z_2$. We know that $M\#N$ have a degree one map onto $M$ by contracting the $N$ portion
in the direct sum to a point. While on the other hand, the double cover of $M\#N$ is $k\mathbb CP^2\# l \overline{\mathbb CP^2}$.
Another more direct example is taking $M\#\mathbb CP^2=(k-1)\mathbb CP^2\#l\overline{\mathbb CP^2}$ and looking at the degree one map from it to $M$. These examples imply that this order would
not be interesting if we include the reducible 
ones into our objects. It only detects the size of the intersection form. This is not the right order in our mind since $k\mathbb CP^2\# l \overline{\mathbb CP^2}$ should be among the simplest ones by their topological types. Notice that any symplectic $4$-manifolds and complex surfaces are irreducible. 

There are three facts worth noting for the mapping order(s). The first one is that for any manifold $M^n$ there is always a 
continuous non-zero degree map onto $S^n$.  
Thus $S^n$ is 
always the minimal one in the order. A fact related to this is that any symplectic/complex $4$-manifold could be realized
as a symplectic/complex ramified cover of $\mathbb CP^2$. So $\mathbb CP^2$ is the minimal  
manifold in the symplectic/complex order. Another fact is the so called Gromov hyperbolization. It says that for 
any manifold $N$, one can find a hyperbolic manifold $H(N)$, such that it maps onto $N$ through a non-zero degree map. In other words, hyperbolic manifold is large with respect to the order.

Now let us assume our manifolds are smooth. Although the mapping order defined using differentiable maps does not give more information than that given by continuous maps, it is indeed give us richer structures when the map is regarding some geometric structures. In a rougher scale, we expect various Kodaira dimensions are compatible with this order, in the sense that
 if $M \succ N$ in a suitable category, then $\kappa(M)\ge \kappa(N)$ for Kodaira dimension $\kappa$ defined in the same category. The three facts mentioned in the previous paragraph are evidences. 
 We have shown in Theorem \ref{mapKod3} that it is indeed true for $\kappa^t$. It is also true for holomorphic Kodaira dimension (see Theorem \ref{ko-dim}).

We expect  the same conclusion is valid for symplectic $4$-manifolds and $(J,J')$ pseudoholomorphic maps (or 
symplectic maps) between them with respect to symplectic Kodaira dimension. 

\begin{question}\label{sdeg}
Suppose that $(M_1, \omega_1)$ and $(M_2, \omega_2)$ are symplectic $4$-manifolds and almost complex structures $J_i$ are tamed by $\omega_i$. If $f$ is a $(J_1,J_2)$-pseudo-holomorphic map ({\it i.e.} $f\circ J_1=J_2\circ f$) of non-zero degree from $(M_1, \omega_1)$ to $(M_2, \omega_2)$, is $\kappa^s(M_1, \omega_1)\ge \kappa^s(M_2, \omega_2)$?
\end{question}

Recall that an almost complex structure $J$ is {\it tamed} by a symplectic form $\omega$ if $\omega(v, Jv)>0$ for any $v\ne 0$.

We could answer this question positively when $\kappa^s(M_1, \omega_1)=-\infty$. 
\begin{prop}
Under the assumptions of Question \ref{sdeg}, if $\kappa^s(M_1, \omega_1)=-\infty$, then $\kappa^s(M_2, \omega_2)=-\infty$.
\end{prop}
\begin{proof}
In this situation, $M_1$ could be covered by $J_1$-holomorphic spheres in certain homology class $A$. As $f$ is onto $M_2$, not all these spheres will be contracted under $f$. This implies the class $f_*A$ is non-trivial. After composing  $f:M_1\rightarrow M_2$, any $J_1$-holomorphic sphere $S^2\rightarrow M_1$ in class $A$ will be a $J_2$-holomorphic sphere in class $f_*A$. All these $J_2$-holomorphic spheres will cover $M_2$ since $f$ is surjective. This implies that the Kodaira dimension $\kappa^s(M_2)=-\infty$ as well.
\end{proof}

For the definition of symplectic Kodaira dimension, we need to reduce the calculations to its minimal model. In the above argument, we apply a birational description, uniruledness, of symplectic manifolds with $\kappa^s=-\infty$. It seems that to prove it in general, we will need a birational description, since it is very hard to regard minimality in the maps.

\subsection{Degree $1$ maps}\label{degree1}
The degree 1 maps are self interesting and are studied extensively in the literatures. 

In dimension three, there is a good formulation of degree one maps using surgeries. First, a well known result says
that any $3$-manifold could be constructed from $S^3$ by a ($\pm1$) surgery along a link, each of whose components 
are unknots (or equivalently we could do a sequence of surgeries along unknots). A result of Boileau and Wang shows 
that any surgery along homotopically trivial knot could be realized by a degree one map. Especially, this gives a proof of
the result that $S^3$ is $1$-dominated by any $3$-manifolds. 

Since surgeries on $3$-manifolds correspond to attaching $2$-handles on $4$-manifolds with boundary, it is natural to have a version of Boileau-Wang's result in this situation as well. Notice the degree of a proper map between compact manifolds with boundary is defined by using the relative cohomology $H^n(M,\partial M; \mathbb Z)$, which is isomorphic to $H_0(M; \mathbb Z)=\mathbb Z$ by Lefschetz duality.

\begin{prop}
Suppose $M$ is a compact $4$-manifold with boundary, and $M'$ is the $4$-manifold obtained by attaching an irreducible $2$-handle $H$ along a homotopically trivial knot $k$ on $\partial M$, then there is a degree one map $f:M'\rightarrow M$. The same statement is true when we attach $2$-handles along a link $l$ with each component homotopically trivial.
\end{prop}

\begin{proof}
Let us first recall the settings of Proposition 3.1 in \cite{BW}. Since $k$ is null-homotopic in $\partial M$, $k$ can be obtained from a trivial know $k'$ by finitely many self-crossing changes of $k'$. Let $D'$ be an embedded disk in $M$ bounded by $k'$. A singular disk $D\subset \partial M$ with $\partial D=k$ is obtained by by identification of pairs of arcs in $D'$ following the self-crossing-changes from $k'$ to $k$. The singular disk $D$ obtained in $\partial M$  with $\partial D=k$ has the homotopy type of a graph. Let $N(D)$ be a regular neighborhood of $D$ in $M$. Then $N(D)$ is an irreducible handlebody. i.e. homeomorphic to $D^4=D^2\times D^2$.  We could make a suitable choice of $N(D)$ such that the attaching region $N(k)\subset N(D)$.

Let us construct a degree one map $f:M'=M-N(k)\cup_{\phi} H\rightarrow M=\overline{M-N(k)}\cup N(k)$. where $H=D^2\times D^2$.

First, the map $f$ at part $M-N(k)$ in $M'$ is defined to be identity. For $\partial M'$, it could be viewed as obtained from a surgery along $k$ from $\partial M$ defined by the map $\phi$. Hence we could define $f$ on this part as in \cite{BW}. Especially, we notice that the part of $\partial H=S^1\times D^2\cup D^2\times S^1$ which is not attached to $M$ is mapped to $N(D)$.

Combining what we said on $M-N(k)$ and on $\partial M'$, the whole boundary $\partial H \subset N(D)$. Since $H$ is $D^4$, we can extend the map to whole $M'$ by sending $H=D^4$ into $N(D)$.

Since $M$ is a compact $4$-manifold, $N(D)$ is a proper subset of $M$, we know the degree of $f$ is one.
\end{proof}

The next step is to analyze the case of maps between closed $4$-manifolds. The $3$-handles and $4$-handles attachings are uniquely determined by the $1$-handles and $2$-handles, especially we know that the union of $3$-handles and $4$-handles will be diffeomorphic to the boundary sum of $m$ $S^1\times D^3$. In particular, we know the $2$-handlebody $X_2$ has boundary $\#_mS^1\times S^2$. Back to $M'$ and $M$ in previous proposition, we have already established the map between the corresponding $2$-handlebodies $M'$ and $M$. If there are no $3$-handles to be attached for both, i.e. $\partial M'=\partial M=S^3$, then the degree one map could be extended to the unique closed-ups of $M'$ and $M$.

The relations of degree $1$ maps with symplectic birational geometry will be discussed in Section \ref{degreecomplex}.

\section{Higher Dimensional Kodaira dimensions and equivalence classes of mapping order}

\subsection{The mapping order for complex and almost complex manifolds}\label{degreecomplex}
For complex surfaces, the Kodaira dimension behaves just as we expect, {\it i.e.} it regards the (meromorphic) mapping order and preserved by covering. Namely, we have the following (see \cite{Ue}).

\begin{theorem}\label{ko-dim}
\begin{itemize}
\item Let $f: M\rightarrow N$ be a generically surjective meromorphic mapping of complex manifolds such that $\dim M=\dim N$. Then we have $\kappa^h(M)\ge \kappa^h(N)$. 

\item Let $f: M\rightarrow N$ be a finite unramified covering of complex manifolds. Then we have $\kappa^h(M)=\kappa^h(N)$.
\end{itemize}
\end{theorem}

The complex projective space $\mathbb CP^n$ is smallest with respect to this mapping order, in the sense that when $M=\mathbb CP^n$,  $N$ has
to be $\mathbb CP^n$ as well.

\begin{example}\label{exotic}
Let $M_1$ be the algebraic surface homeomorphic but not diffeomorphic to $\mathbb CP^2\# 5\overline{\mathbb CP^2}$ constructed
in \cite{PPS}. We know that there are differentiable degree one maps from each direction because we can homotope the homeomorphism from both directions. Theorem
\ref{ko-dim} tells us that there is no non-trivial holomorphic map $f:\mathbb CP^2\# k\overline{\mathbb CP^2}\rightarrow M_1$. However, we have $f:M_1\# k\overline{\mathbb CP^2}\rightarrow 
\mathbb CP^2\# 5\overline{\mathbb CP^2}$ with $k\ge 5$. There is no such map for $k<5$ since they are simply connected which is ordered by their intersection forms by \cite{DW}.

\end{example}

Notice that Theorem \ref{ko-dim} is related to (and could be viewed as $0$-dimensional generalization of) the Iitaka conjecture, which states
that a fiber space $f:X\rightarrow Z$ satisfies $\kappa^h(X)\ge \kappa^h(Z)+\kappa^h(F)$ where $F$ is a general fiber of $f$. Here,
an (analytic) fiber space is a proper surjective morphism with connected fibres. 
Actually, the Iitaka conjecture is one of the main motivations for our additivity principle of Kodaira dimensions.

Furthermore, the mapping order also regards other invariants. Recall that  the algebraic dimension $a(M)$ of a complex manifold is defined as the transcendence degree over $\mathbb C$ of the field $\mathbb C^{Mer}(M)$ of meromorphic functions. 
When $f: M \rightarrow N$ is a surjective holomorphic map,  the algebraic dimensions  $a(M)=a(N)$ (see \cite{Ue}).

In the almost complex setting, it is worth noting that a $(J,J')$ holomorphic map makes the $J$-anti-invariant cohomology
dimension $h_J^-$, which is introduced in \cite{LZ07, DLZ09}, non-decreasing. Let $(M^{2n}, J)$ be an almost complex manifold. The almost complex
structure acts on the bundle of real 2-forms $\Lambda^2$ as an
involution, by $\alpha(\cdot, \cdot) \rightarrow \alpha(J\cdot,
J\cdot)$. This involution induces the splitting into $J$-invariant, respectively,
$J$-anti-invariant 2-forms
$$\Lambda^2=\Lambda_J^+\oplus \Lambda_J^-.$$
We denote by $\Omega^2$ the space of 2-forms on $M$
($C^{\infty}$-sections of the bundle $\Lambda^2$), $\Omega_J^+$ the
space of $J$-invariant 2-forms, {\it etc.} Let also $ \mathcal Z^2$ denote
the space of closed 2-forms on $M$ and let $\mathcal Z_J^{\pm} =
\mathcal Z^2 \cap \Omega_J^{\pm}$.
Then 
$$H_J^{\pm}(M)=\{ \mathfrak{a} \in H^2(M;\mathbb R) | \exists \;
\alpha\in \mathcal Z_J^{\pm} \mbox{ such that } [\alpha] =
\mathfrak{a} \}.$$
The dimensions $\dim H_J^{\pm}(M)$ is denoted as $h_J^{\pm}(M)$.

\begin{prop}\label{hJ-}For a general non-zero degree map $f: M \rightarrow N$, 
$b_2^{\pm}(M)\ge b_2^{\pm}(N)$.
If there is a surjective equi-dimensional $(J,J')$ holomorphic map $f: M \rightarrow N$ for two almost complex manifolds $(M,J)$ and $(N,J')$, then 
$h_J^{\pm}\ge h_{J'}^{\pm}$. 
\end{prop}
\begin{proof}

We only show it for $h_J^-$. The argument for $h_J^+$ is similar and the conclusion for Betti numbers are well-known. Recall when $\alpha$ is a $J'$-anti-invariant two form on $N$, $$\alpha(J'X,J'Y)=-\alpha(X,Y).$$

Now $f^*\alpha$ is a two form on $M$, and $$f^*\alpha(JX, JY)=\alpha(J'f_*X, J'f_*Y)=-\alpha(f_*X, f_*Y)=-f^*\alpha(X, Y).$$

Because $f$ has non-zero degree, $f^*\alpha$ is non-trivial and $$(f^*\alpha)^n=\deg(f)\cdot \alpha^n.$$

Finally, $f^*$ commutes with the differential $d$. Hence $h_J^-\ge h_{J'}^-$.
\end{proof}

There is no such example with $h_J^->h_{J'}^-$ coming into the author's mind when $f: (M, J) \rightarrow (N, J')$ is a  surjective equi-dimensional $(J,J')$ holomorphic map.

Now, let us look at degree $1$ maps. First, since all birational maps are of degree $1$, it is of interests to understand the relation of it with symplectic birational geometry (see \cite{LR}). However, a plain differentiable degree $1$ map will not preserve the birational class starting from dimension $4$. 
Boileau and Wang \cite{BW} proves that
any $3$-manifold $M$ is 1-dominated by a hyperbolic manifolds which is meanwhile a surface bundle $H(M)$. Thus $H(M)
\times S^1$ again 1-dominates $M \times S^1$. Once $M$ is a surface bundle, both could be endowed with a symplectic structure. Hence, a degree one map could change (complex/symplectic) Kodaira dimension
of manifolds, thus the symplectic birational equivalence class. Symplectic fiber sum construction provides more such examples. Moreover, Example \ref{exotic} provides a simply connected example. Hence, one has to impose more conditions on the map in addition to its degree in order to preserve the birational equivalence. A natural question (as mentioned to the author by Tian-Jun Li) is

\begin{question}\label{JJ'holomorphic}
Suppose $f: M_1\rightarrow M_2$ is a $(J_1, J_2)$ pseudo holomorphic map of degree $1$, where $J_1$ and $J_2$ are almost complex structures tamed by symplectic structures $\omega_1$ and $\omega_2$. Is the map a composition of blow downs?
\end{question}
Blow-downs compatible with $J_i$ are very rigid objects. This  excludes a lot of possibilities for the target (see \cite{DLZQJM}).

We could show that the map is birational if it is a holomorphic map. Let $f: M\rightarrow N$ be a degree one holomorphic map, then except possibly
the set $D$ where $\det(\frac{\partial z_i}{\partial w_j})=0$, other parts are $1:1$. Here $z_i$ and $w_j$ are homomorphic local coordinates of $N$ and $M$ respectively. 
This zero locus $D$ is a complex subvariety of complex codimension one in $M$. Thus $M$ and $N$ are birational. When both $M$ and $N$ are projective varieties, the birational morphism is factored as several blowdowns. 

For general almost complex structures, 
it is apparent that the Jacobi matrix $Df(x)$ of $f$ in any point $x\in M_1$ is positive definite. Moreover, the degree is calculated locally as the sum of the signs of the determinant of Jacobian at each preimage. 
Hence if $f: M_1\rightarrow M_2$ is a finite ({\it i.e.} the preimage of any point is a finite set) $(J_1, J_2)$ pseudo holomorphic map of degree $1$, then $f$ is a diffeomorphism.

More generally, we would like to know whether it is true that any non-zero degree holomorphic map is homotopic to a composition of blow-downs and branched 
covering? One evidence is that every algebraic surface could be realized as a branched covering of $\mathbb CP^2$.
Another notable fact is that for any equi-dimensional dominating morphism $f: M\rightarrow N$,  we have the ramification
formula $K_M=f^*K_N+R_f$, where the effective $\mathbb Q$-divisor $R_f\ge 0$ is called the ramification divisor. We expect the ramification formula would still hold for $(J, J')$-pseudoholomorphic maps. This  would help us to understand Question \ref{sdeg}. 

Finally, it is amusing to look at the case when the degree is zero. Let us first suppose the complex
dimension is one, then the Liouville's theorem tells us that any such map is a constant map, i.e. it maps onto a point in 
$\mathbb CP^1$. This statement could be generalized to any genus target. This is because any non-compact Riemannian surface is
Stein and further we know any Stein manifold could be biholomorphically embedded into $\mathbb C^N$. For higher dimensions,  if $f:M\rightarrow N$ is of degree zero, then $f$ 
maps into a proper subvariety of $N$.

\subsection{Higher dimensional Kodaira dimensions and additivity}
For higher (even) dimensions, we still have the definition of Kodaira dimensions for complex manifolds. However, it is not known if there is a suitable generalization of $4$-dimensional symplectic Kodaira dimension. One difference between dimensions no more than four and higher is that for dimension larger than $4$ the Kodaira dimension relies not only the smooth structures on the manifold, but also the complex structure on it. 

The following example is due to R\u{a}sdeaconu, but may not be that well-known. Hence we reproduce it here for the convenience of readers. Let us take $M=\mathbb CP^2\# 8\overline{\mathbb CP^2}$ and $N$ is the Barlow surface \cite{Bar} which is a complex surface of general type homeomorphic to $M$. One can also take $M=\mathbb CP^2\# 5\overline{\mathbb CP^2}$ and $N$ is the complex surface of general type homeomorphic to $M$ constructed in \cite{PPS}. Then $M \times \Sigma_g$ is diffeomorphic to $N \times \Sigma_g$. This is because that, at first, they are $h-$cobordant because $M$ and $N$ are so. Second, they are $s-$cobordant because the Whitehead group $Wh(M\times \Sigma_g)=Wh(N\times \Sigma_g)=Wh(\Sigma_g)=0$. Then by the $s-$cobordism theorem proved independently by Mazur, Stallings, and Barden, they are diffeomorphic to each other. On the other hand, they are not the same as complex manifolds since they have distinct Kodaira dimensions. The Kodaira dimension $\kappa^h(M\times \Sigma_g)=-\infty$ as $\kappa^h(M)=-\infty$. However, $\kappa^h(N\times \Sigma_g)=2$ when $g=1$ and $\kappa^h(N\times \Sigma_g)=3$ when $g>1$. In \cite{R}, there are various examples of diffeomorphic manifolds with different Kodaira dimensions constructed.  It is worth noting that none of them is simply connected.

In \cite{LR}, Li and Ruan propose a possible way to define symplectic Kodaira dimension in dimension $6$. First let us recall the following definition.

\begin{definition}
A symplectic $6$-manifold is minimal if it does not contain any rigid stable uniruled divisor.
\end{definition}

Here a (symplectic) uniruled divisor is nothing but a rational or ruled $4$-manifold. A uniruled divisor is {\it stable} if one of its uniruled classes $A$ has a nontrivial Gromov-Witten invariant of the ambient manifold with $K_{\omega}(A)\le -1$. A uniruled divisor is {\it rigid} if none of its uniruled class is uniruled in the ambient manifold.

This definition only takes care about the divisorial contraction. However, for algebraic $3$-folds, flip or small contraction cannot happen in smooth category.

Assume $(M, \omega)$ is a minimal symplectic manifold of dimension $6$, then Li and Ruan propose the following definition of simplectic Kodaira dimension:

\begin{equation}\label{liruan}
\kappa^s(M,\omega)=\left\{\begin{array}{cc}
-\infty & \hbox{if one of $K_{\omega}^i\cdot [\omega]^{3-i}<0$},\\
k& \hbox{ if $K_{\omega}^i\cdot [\omega]^{3-i}=0$ for $i> k$ and $K_{\omega}^i\cdot [\omega]^{3-i}>0$ for $i\le k$}.\\
\end{array}\right.
\end{equation}
 
There is an issue of well definedness. For example, one cannot yet exclude the possibility of a minimal symplectic $6$-manifold $(M, \omega)$ with \begin{equation}\label{bad}K_{\omega}\cdot [\omega]^2=0,K_{\omega}^2\cdot [\omega]=0\  \hbox{but}\  K_{\omega}^3>0,\end{equation}
although there are no counterexample in the author's sight as well.

In dimension $4$, a similar issue is resolved only with the help of Seiberg-Witten invariant which has no counterpart in higher dimensions.

For product symplectic manifolds of type $M^4\times \Sigma_g$ with $g\ge 0$ and product symplectic forms, we can verify that the proposed definition is good in the sense that no bad cases like \eqref{bad} would happen. It also satisfies the additivity and is compatible with complex Kodaira dimension. 

\begin{prop}\label{proposed} Suppose $(M^4, \omega_M)$ is a symplectic $4$-manifold and $(M^4\times \Sigma_g, \omega_M \times \omega_g)$ is minimal, then the following additivity relation holds $$\kappa^s(M^4\times \Sigma_g, \omega_M \times \omega_g)=\kappa^s(M^4)+\kappa^s(\Sigma_g).$$

Moreover, when $M^4$ admits a complex structure $J$, $$\kappa^s(M^4\times \Sigma_g, \omega_M \times \omega_g)=\kappa^h(M^4\times \Sigma_g, J\times j).$$

\end{prop}

\begin{proof}

First notice that minimality of a $4$-dimensional symplectic manifold only depends on the diffeomorphism type. So we will say $M^4$ is minimal instead of saying $(M^4, \omega_M)$ is so.

If $M^4$ is not minimal with $E$ as an exceptional curve, then $D=E\times \Sigma_g$ is a rigid stable uniruled divisor. Here $A=[E]$ is a uniruled class in $D$ and $K_{\omega_M\times \omega_g}(A)=-1$.

Hence $M^4$ is minimal. Notice that the canonical class $K=K_{M\times \Sigma_g}=K_M+K_{\Sigma_g}$ and $[\omega]=[\omega_M]+[\omega_g]$.  We calculate  $$K^3=K_M^2\cdot K_{\Sigma_g},$$ 
$$K^2 \cdot [\omega]=
K^2_M \cdot [\omega_g]+2K_M\cdot [\omega_M]\cdot K_{\Sigma_g},$$
$$K \cdot [\omega]^2=
[\omega_M]^2 \cdot K_{\Sigma_g}+ 2[\omega_M]\cdot K_M\cdot [\omega_g].$$

If $g=0$, then we need to prove $\kappa^s(M^4\times S^2, \omega_M \times \omega_0)=-\infty$, or one of the products $K\cdot [\omega]^2,K^2\cdot [\omega],  K^3$ is negative. Notice $K_{S^2}=-2<0$. Both $K^3\ge 0, K^2\cdot [\omega]\ge 0$ would imply $K^2_M\le 0, K_M\cdot [\omega_M]\le 0$. However, this would imply $K\cdot [\omega]^2<0$.

If $g=1$, then $K_{\Sigma_1}=0$ and  $K^3=0$. And the signs of $K^2_{\omega}\cdot [\omega]$ and $K\cdot [\omega]^2$ are determined by that of $K^2(M)$ and $K_M\cdot [\omega_M]$ respectively. Hence $\kappa^s(M^4\times \Sigma_g, \omega_M \times \omega_g)=\kappa^s(M^4)$.

If $g\ge 2$, then $K_{\Sigma_g}>0$ and $\kappa^s(\Sigma_g)=1$. If $\kappa^s(M)\ge 0$, then $K\cdot [\omega]^2>0$, $K^3\ge 0, K^2\cdot [\omega]\ge 0$. Furthermore, $K^3=0$ if and only if $K^2_M=0$, i.e. $\kappa^s(M)=0$ or $1$. And in addition $K^2\cdot [\omega]= 0$ if and only if $\kappa^s(M)=0$. This verifies $\kappa^s(M^4\times \Sigma_g, \omega_M \times \omega_g)=\kappa^s(M^4)+1$ when  $\kappa^s(M)\ge 0$.

When $\kappa^s(M)=-\infty$, we want to show one of the product $K_{\omega}\cdot [\omega]^2,K_{\omega}^2\cdot [\omega],  K_{\omega}^3$ is negative. Actually, we will show it is always true for any (possibly non-minimal) rational or ruled $4$-manifold. If $K^2_M<0$, then $K^3<0$. If $K^2_M=0$ and $K_M\cdot [\omega_M]<0$, then $K^2\cdot [\omega]<0$. Hence we could assume $K^2_M> 0$ and $K_M\cdot [\omega_M]<0$. The only possibilities are $M=\mathbb CP^2\# k\overline{\mathbb CP^2}$ when $k<9$ and $S^2\times S^2$.  For all these cases, if we let $K_M=3H+\sum_i E_i$ for $\mathbb CP^2\# k\overline{\mathbb CP^2}$ and $K_M=-2S_1-2S_2$ for $S^2\times S^2$, then $\omega$ can be written as $aH-\sum_i b_iE_i$ and $aS_1+bS_2$ with all the coefficient positive.
If both $K^2 \cdot [\omega]$ and $K \cdot [\omega]^2$ are non-negative, then we have $$K_{M}^2\cdot [\omega_M]^2\geq 4(K_{M} \cdot [\omega_M])^2$$ from our formula for $K^2 \cdot [\omega]$ and $K \cdot [\omega]^2$.
It is straightforward to check that the inequality is impossible for all the above cases, i.e. one of $K^2 \cdot [\omega]$ and $K \cdot [\omega]^2$ has to be negative. This completes the proof of our first statement.

The second statement follows from the facts that $\kappa^s(M^4)=\kappa^h(M^4)$ and that the complex Kodaira dimension is additive for the product complex structure, i.e. $\kappa^h(M\times \Sigma_g, J\times j)=\kappa^h(M, J)+\kappa^h(\Sigma_g, j)$.
\end{proof}

\subsection{Equivalence classes, Entropy and Gromov norm}
In last section, we mention that the order defined  through non-zero degree maps may not be a partial order. This raises a question that what is each equivalent class like. Precisely, for which manifolds $M$ and $N$, we could have non-zero degree maps $f: M\rightarrow N$
and $g: N\rightarrow M$? Especially, can we find some invariants for each class?

There is another mapping order, defined in a similar manner but through degree one maps (instead of general non-zero degree maps). In dimension three, the degree one mapping order is indeed a partial order \cite{Rong}, 
{\it i.e.} each equivalence class contains exactly one manifold. This is because $3$-manifolds are almost determined by their fundamental groups and
the fundamental groups of $3$-manifolds are residually finite and thus are Hopfian. In other words, it is more or less a group theory reasoning. However, when dimension is getting 
higher,  it does not give rise a partial order when we identify two manifolds if they are homeomorphic. This is because there are examples of homotopy equivalent but not homeomorphic manifolds, {\it e.g.} certain lens spaces. But on the other hand, the homotopy equivalent manifolds $1$-dominate each other by definition. 

It is then natural to ask that what will happen if we identify two manifolds in the same equivalence class of degree one mapping order when they are merely homotopy equivalent. Let us consider aspherical closed oriented $n$-manifolds with Hopfian fundamental groups. By the same argument as Rong's for $3$-manifolds, if such $M$ and $N$ $1$-dominate each other, then they are homotopy equivalent. If Borel conjecture holds, they are homeomorphic to each other. Recall the Borel conjecture: Let $M$ and $N$ be closed and aspherical topological manifolds, if they are homotopy equivalent, then they are homeomorphic to each other. 

In addition, Gromov norm is also an invariant for the equivalence classes for the  degree one mapping order. This is because $||M||\ge \deg(f)\cdot ||N||$ if there is a map $f:M\rightarrow N$. It is also amusing to note that if we use the original mapping order, then each equivalence class is 
either the same as the one given by degree one order, or each manifold in the equivalence class has Gromov norm $0$. 
 Especially, by Gromov's proof of Mostow rigidity, the equivalence class of degree one mapping order containing a hyperbolic manifold is exactly this manifold. This is because, by above discussion, any two manifolds in an equivalence class would have the same Gromov norm. Moreover, any degree $1$ map between hyperbolic manifolds with same volume is a homotopy equivalence and thus an isometry. 

 One advantage to consider this degree one mapping order is 
that we may prevent the issue of reducibility as indicated in LeBrun's example. 
Another advantage of this mapping order is that 
 the set of topological entropies would be an invariant for each equivalence class.
The {\it Shub topological entropy} $S(f)$ of a map $f: N\rightarrow N$ is defined as $\log\lambda(f)$, where $\lambda(f)$ is the maximal spectral radius
$f_*: H_l(N, \mathbb R)\rightarrow H_l(N, \mathbb R)$ among all $l$. 

\begin{prop}
Assume $M$ and $N$ are equivalent
through degree one map, i.e. there are $g:M\rightarrow N$ and $h:N \rightarrow M$ both of degree one. Then, for any 
map $f_1:M \rightarrow M$, the composed map $f_2=g\circ f_1\circ h: N\rightarrow N$ would have $S(f_2)=S(f_1)$. 
\end{prop}
\begin{proof}
This follows from Nakayama's lemma. One corollary of it says the following (see Theorem 2.4 of \cite{Mat}).

\begin{lemma}\label{nakayama}
Suppose $R$ is a commutative ring. If $Z$ is a finitely generated $R$-module and $f: Z\rightarrow Z$ is a surjective endomorphism, then $f$ is an isomorphism.
\end{lemma}

In our situation, homology groups are finitely generated $\mathbb Z$-modules. Since degree one map $g:M\rightarrow N$ (resp. $h:N \rightarrow M$) would induce epimorphisms $g_*: H_k(M)\rightarrow H_k(N)$ (resp.  $h_*: H_k(N)\rightarrow H_k(M)$). See for example Lemma 1.2 in \cite{Rong}. Hence the compositions $g_*\circ h_*$, $h_*\circ g_*$ are epimorphisms, and thus isomorphisms by Lemma \ref{nakayama}. This implies $H_k(N) \cong H_k(M)$ and $g_*$, $h_*$ are isomorphisms. Hence the Shub topological entropy $S(f_2)=S(f_1)$.
\end{proof}

Notice that entropy invariant (and its variants) could make complementary use with Gromov norm. The
latter detects hyperbolic pieces and the entropy sees the others because if one admits a self degree $>1$ map then Gromov norm
has to be $0$.

Gromov norm has other interesting applications in the $4$-manifolds theory. For example, we know that  $\mathbb CP^2\# k
\overline{\mathbb CP^2}$ and $S^2 \times S^2$ all have Gromov norm $0$ (even for each homology class of them). This implies that
the exotic differential structures of them will not admit any metric with negative sectional curvature. 
  Another example is that any manifold with amenable
fundamental group ({\it e.g.} trivial, nilpotent, solvable, abelian...) will not be greater than the ones with negative sectional curvature under the
mapping order. 
Moreover, we have the following, which is surely known to the experts.
\bigskip 

{\it  There does  not exist two homotopic closed Riemannian manifolds such that one has negative sectional 
curvature  and the other has
\begin{enumerate}
\item non-negative Ricci curvature; or
\item almost non-negative sectional curvature.
\end{enumerate}
}

\bigskip

\begin{proof}
By Cheeger-Gromoll \cite{CG}, if a manifold admits a metric with non-negative Ricci curvature, then its
fundamental group is virtual abelian, {\it i.e.} there is an abelian subgroup of it with finite index.
On the other hand, a group is amenable if it has a finite index amenable subgroup and an abelian group is amenable, so the 
fundamental group of a manifold with non-negative section curvature has to be amenable. 
Similarly, if a manifold admits almost non-negative sectional curvature ({\it i.e.} it admits a sequence of Riemannian metrics $\{g_n\}_{n\in \mathbb N}$ whose sectional curvatures and diameters satisfy $\sec(M, g_n)\ge -\frac{1}{n}$ and $\text{diam}(M, g_n)\le \frac{1}{n}$), then the  fundamental groups of such manifolds are virtually nilpotent \cite{FY}, which is also amenable.
This implies the Gromov norm of such
a manifold is zero. Therefore, it cannot admit any metric with negative sectional curvature because Gromov norm of such a 
manifold would have nonzero Gromov norm \cite{IY}.
\end{proof}

Notice that negative sectional curvature cannot be replaced by negative Ricci curvature because we know that any manifold of dimension greater than two could admit a metric with negative Ricci curvature \cite{Lohk}.

For non-closed manifolds, the situation is different: there are homotopy types of manifolds, {\it e.g.} $\mathbb R^n$, which admit complete metrics of non-negative and negative sectional curvatures respectively. In general, Cheeger-Gromoll soul takes care of the complete non-negative metrics. On the other hand, the classical Hadarmard-Cartan theorem says that if $M^n$ is a connected complete Riemannian manifold with non-positive
sectional curvature, then its universal covering space is diffeomorphic to $\mathbb R^n$. It is probably true that  if
a closed manifold admits both non-positive and non-negative sectional curvature, then it has to be flat. All closed flat $n$-manifolds are finitely covered by
$T^n$.

\end{document}